\documentclass[11pt]{amsart}
\usepackage{amsmath} 
\usepackage{amsthm,amsfonts,amssymb,mathrsfs,amscd,amstext,amsbsy} 
\usepackage{epic,eepic} 
\usepackage{yfonts}
\usepackage{paralist,enumerate}
\usepackage[all]{xy}
\usepackage{hyperref}
\hypersetup{colorlinks}
\usepackage{tikz}
\usepackage{graphicx,subcaption}
\usetikzlibrary {positioning}

\newtheorem{theorem}{Theorem}[section]
\newtheorem{cor}[theorem]{Corollary}

\newtheorem{theo}[theorem]{Theorem}
\newtheorem{lem}[theorem]{Lemma}
\newtheorem{pro}[theorem]{Proposition}
\newtheorem{rem}[theorem]{Remark}
\newtheorem{exa}[theorem]{Example}

\newtheorem{Definition}[theorem]{Definition}
\newtheorem*{Definition*}{Definition}

\def\qed{\hfill \ifhmode\unskip\nobreak\fi\quad\ifmmode\Box\else$\Box$\fi\\ }

\begin{document}

\title[Unitary $S^1$-manifolds with discrete fixed point sets]{Circle actions on unitary manifolds with discrete fixed point sets}
\author{Donghoon Jang}
\address{Department of Mathematics, Pusan National University, Pusan, Korea}
\email{donghoonjang@pusan.ac.kr}
\thanks{MSC 2020: 58C30, 58J20}
\thanks{Keywords: unitary manifold, circle action, stable complex manifold, fixed point, weight, multigraph, Chern class}
\thanks{Donghoon Jang was supported by the National Research Foundation of Korea(NRF) grant funded by the Korea government(MSIT) (2021R1C1C1004158).}
\begin{abstract}
In this paper, we prove various results for circle actions on compact unitary manifolds with discrete fixed point sets, generalizing results for almost complex manifolds. For a circle action on a compact unitary manifold with a discrete fixed point set, we prove relationships between weights at the fixed points. As a consequence, we show that there is a multigraph that encodes the fixed point data (a collection of multisets of weights at the fixed points) of the manifold; this can be used to study unitary $S^1$-manifolds in terms of multigraphs. We derive results regarding the first equivariant Chern class, obtaining a lower bound on the number of fixed points under an assumption on a  manifold. We determine the Hirzebruch $\chi_y$-genus of a compact unitary manifold admitting a semi-free $S^1$-action, and obtain a lower bound on the number of fixed points. 
\end{abstract}
\maketitle

\section{Introduction} \label{s1}

The purpose of this paper is to prove various results for circle actions on compact unitary manifolds with discrete fixed point sets, generalizing results for circle actions on almost complex manifolds. Any complex manifold or any symplectic manifold is almost complex, and any almost complex manifold is unitary. Our results therefore also generalize results for complex manifolds and symplectic manifolds. 

Consider an action of the circle group $S^1$ on a manifold. Then the dimension of the manifold and the dimension of any fixed component have the same parity. Thus, if the fixed point set is non-empty and discrete, because the fixed point set is $0$-dimensional, the dimension of the manifold is necessarily even; if in addition the manifold is compact, then the number of fixed points is finite. Therefore, unitary $S^1$-manifolds with (non-empty) discrete fixed point sets that we consider in this paper are even dimensional, while a general unitary manifold need not have even dimension.

A fundamental difference between unitary $S^1$-manifolds and almost complex $S^1$-manifolds for our generalizations is that for a unitary $S^1$-manifold some sign issue occurs at each fixed point, depending on two orientations on the tangent space at the fixed point; see Section \ref{s2} for detailed explanation. Our generalizations are done by carefully handling this sign issue.

We generalize several results on relationships between weights at the fixed points for almost complex $S^1$-manifolds to unitary $S^1$-manifolds. These are Proposition \ref{p31}, Lemma \ref{l33}, and Corollary \ref{c36}. To illuminate, let the circle group $S^1$ act on a $2n$-dimensional compact unitary manifold $M$ with a discrete fixed point set. From Proposition \ref{p31}, whose statement and proof are a bit technical, we prove Lemma \ref{l33} that for any integer $w$,
\begin{center}
$\displaystyle \sum_{p \in M^{S^{1}}} \epsilon(p)  \cdot N_{p}(w)=\sum_{p \in M^{S^{1}}} \epsilon(p) \cdot N_{p}(-w)$,
\end{center}
where $N_{p}(w)=|\{j \, | \, w_{p,j}=w, 1 \leq j \leq n\}|$ is the number of times $w$ occurs as weights at $p$ and $\epsilon(p)$ is either $+1$ or $-1$, depending on two orientations on the tangent space $T_pM$ at $p$. Consequently, we obtain Corollary \ref{c36} that
\begin{center}
$\displaystyle \sum_{p \in M^{S^1}} \epsilon(p) \cdot c_1(M)(p)=0$,
\end{center}
where $c_1(M)(p)$ is the first equivariant Chern class of $M$ at $p$. Letting $\epsilon(p)=+1$ for all fixed point $p$ recovers corresponding results for almost complex manifolds.

Using the above relationships, we show that as for an almost complex $S^1$-manifold, for a unitary $S^1$-manifold there exists a multigraph that encodes the fixed point data of the manifold, a collection of multisets of the weights at the fixed points; see Definition \ref{d2} and Proposition \ref{p312}. Moreover, the multigraph contains information on which fixed points are in the same isotropy submanifold; see (4) of Definition \ref{d2}. On the other hand, a multigraph for a unitary $S^1$-manifold assigns every vertex a sign $+1$ or $-1$, while a multigraph for an almost complex $S^1$-manifold does not have such an issue; the sign of every vertex is $+1$. Such a multigraph can be used to study a unitary $S^1$-manifold, and we illustrate this in Examples \ref{ex1}, \ref{ex2}, \ref{ex3}, and \ref{ex4}.

For a symplectic circle action on a compact symplectic manifold $M$, if the Chern class map is somewhere injective (see the definitions in Section \ref{s3}), Pelayo and Tolman proved that there are at least $\frac{1}{2} \dim M+1$ fixed points \cite{PT}. We note that the same proof applies to a unitary manifold, if we take accout of the sign issue that occurs to each fixed point; see Lemma \ref{l37} and Theorem \ref{t38}. 

For a semi-free circle action on a compact almost complex manifold $M$ with a discrete fixed point set, Tolman and Weitsman proved that there are exactly $\textrm{Todd}(M) \cdot 2^{\frac{\dim M}{2}}$ fixed points \cite{TW}. We extend the result to unitary manifolds; if $\textrm{Todd}(M) \neq 0$, then there are at least $|\textrm{Todd}(M)| \cdot 2^{\frac{\dim M}{2}}$ fixed points, see Theorem \ref{t39}. For a unitary manifold, however, on any even dimension there is an example of a semi-free action with 2 fixed points, a rotation of the $2n$-sphere $S^{2n}$; take $a_i=1$ for all $i$ in Example \ref{ex3}.

We explain the structure of this paper. In Section \ref{s2}, we review necessary background and preliminary results, and derive a formula (Theorem \ref{t21}) for a unitary $S^1$-manifold with a discrete fixed point set, which follows from the rigidity of the equivariant $T_{x,y}$-genus proved by Krichever \cite{Kr}, Theorem \ref{t20}. Section \ref{s3} is the main part of this paper, in which we prove the aforementioned results on circle actions on compact unitary manifolds with discrete fixed point sets.

The author would like to thank the anonymous referee for helpful comments and suggestions, which help improve the exposition of this paper.

\section{Background and preliminaries} \label{s2}

A \textbf{unitary structure} (also called a \textbf{stable complex structure} or a \textbf{weakly almost complex structure}) on a manifold $M$ is a complex structure on the vector bundle $TM \oplus \underline{\mathbb{R}}^k$ for some non-negative integer $k$, where $TM$ is the tangent bundle of $M$ and $\underline{\mathbb{R}}^k$ denotes the $k$-dimensional trivial bundle over $M$. A \textbf{unitary manifold} (also called a \textbf{weakly almost complex manifold}) is a manifold with a unitary structure on it. An \textbf{almost complex structure} on a manifold is a bundle map on its tangent bundle, which restricts to a linear complex structure on each tangent space. An \textbf{almost complex manifold} is a manifold with an almost complex structure on it. By definition, any almost complex manifold is unitary but a unitary manifold need not admit an almost complex structure. For instance, for any $n$ the $n$-sphere $S^n$ admits a unitary structure since $TS^n \oplus \underline{\mathbb{R}}^n = S^n \times \mathbb{R}^{2n}=S^n \times \mathbb{C}^n$, but among the $n$-spheres only the 2-sphere $S^2$ and the 6-sphere $S^6$ admit almost complex structures.

Any unitary manifold is oriented in a natural way as follows. Let $M$ be a unitary manifold; suppose that there is a complex structure on the bundle $TM \oplus \underline{\mathbb{R}}^k$ for some non-negative integer $k$. Then the complex structure on the bundle $TM \oplus \underline{\mathbb{R}}^k$ induces an orientation on it, and the bundle $\underline{\mathbb{R}}^k$ is oriented in the usual way. The orientation on $TM \oplus \underline{\mathbb{R}}^k$ and the orientation on $\underline{\mathbb{R}}^k$ then determine an orientation of $TM$ (of $M$) and hence an orientation on the tangent space $T_mM$ for each $m \in M$, which we record as the first orientation on $T_mM$.

For an action of a group $G$ on a manifold $M$, denote by $M^G$ its fixed point set, that is, 
\begin{center}
$M^G=\{m \in M \, | \, g \cdot m=m \textrm{ for all }g \in G\}$.
\end{center}
If $H$ is a subgroup of $G$, $H$ acts on $M$; we denote the $H$-fixed point set by $M^H$ accordingly.

Let the circle group $S^1$ act on a unitary (almost complex) manifold $M$. Let $J$ denote the unitary (almost complex) structure on $TM \oplus \underline{\mathbb{R}}^k$ (on $TM$). The action on $M$ induces an action on its tangent bundle $TM$. We say that the action \textbf{preserves the unitary (almost complex) structure} $J$ if $J \circ dg=dg \circ J$ for all $g \in S^1$, where the circle acts on $\underline{\mathbb{R}}^k$ trivially for $M$ unitary. Alternatively, we say that $J$ is an equivariant unitary (almost complex, respectively) structure.

Let the circle act on a unitary manifold $M$. Throughout this paper, any circle action on a unitary manifold is assumed to preserve the unitary structure. Let $p$ be an isolated fixed point. Because the dimension of $M$ and the dimension of any fixed component have the same parity and an isolated fixed point has dimension 0, the dimension of $M$ is necessarily even; let $\dim M=2n$.
Applying Lemma \ref{l23} below with taking $G=H=S^1$ and $F=p$, the equivariant unitary structure on $M$ induces a fiberwise complex structure on the normal space $N_pM$ of $p$, which is equal to the tangent space $T_pM$, such that $S^1$ acts on $N_pM$ by complex bundle automorphisms. Then the tangent space $T_pM(=N_pM)$ to $M$ at $p$ is a complex vector space that decomposes into the sum of $n$ complex 1-dimensional spaces
\begin{center}
$\displaystyle T_pM=\oplus_{i=1}^n L_{p,i}$
\end{center}
for some $n$, where on each $L_{p,i}$ the circle acts as multiplication by $g^{w_{p,i}}$ for all $g \in S^1$, for some non-zero integers $w_{p,i}$, $1 \leq i \leq n$. This also shows that the dimension of $M$ is even. The non-zero integers $w_{p,1},\cdots,w_{p,n}$ are called the \textbf{weights} at $p$. The \textbf{index} $n_{p}^M$ of $p$ is the number of negative weights at $p$. We shall use $n_p$ instead of $n_{p}^M$ if there is no confusion. The complex structure on $T_pM$ induces an orientation of $T_pM$, which we record as the second orientation of $T_pM$.

We have discussed that the tangent space $T_pM$ has two orientations, the first one induced from the orientation of $M$ ($TM$) and second one induced from the complex structure on $T_pM$. We define the \textbf{sign} of $p$, denoted by $\epsilon_M(p)$, to be $+1$ if the two orientations on $T_pM$ agree, and $-1$ otherwise. If there is no confusion, we shall use $\epsilon(p)$ for $\epsilon_M(p)$, omitting $M$.

Let $M$ be a compact unitary manifold. Let $T_{x,y}(M)$ be the $T_{x,y}$-genus of $M$, where $T_{x,y}$-genus is the genus belonging to the power series 
\begin{center}
$\displaystyle \frac{u(xe^{u(x+y)}+y)}{e^{u(x+y)}-1}$.
\end{center}
Suppose that the circle acts on $M$ with a non-empty discrete fixed point set. Then the dimension of $M$ is even; let $\dim M=2n$. In \cite{Kr}, Krichever proved that the equivariant $T_{x,y}$-genus of $M$ is rigid; it is independent of the choice of an element of $S^1$ and is equal to the $T_{x,y}$-genus of $M$. By the Krichever's result, the following formula holds for all indeterminates $z$.

\begin{center}
$\displaystyle T_{x,y}(M)=\sum_{p \in M^{S^1}} \epsilon(p) \cdot \prod_{i=1}^n \frac{xz^{w_{p,i}}+y}{z^{w_{p,i}}-1}$ 
\end{center}
Letting $t=z^{-1}$ and rewriting, we obtain the following formula.

\begin{theorem} \label{t20} \cite{Kr} Let the circle act on a $2n$-dimensional compact unitary manifold $M$ with a discrete fixed point set. Then the $T_{x,y}$-genus of $M$ satisfies
\begin{center}
$\displaystyle T_{x,y}(M)=\sum_{p \in M^{S^1}} \epsilon(p) \cdot \prod_{i=1}^n \frac{x+y \cdot t^{w_{p,i}}}{1-t^{w_{p,i}}}$
\end{center}
for all indeterminates $t$. 
\end{theorem}

For a compact unitary manifold $M$, write $T_{x,y}(M)=\sum_{i=0}^n T^i \cdot x^i \cdot y^{n-i}$ for some integers $T_i$, $0 \leq i \leq n$. If $x=1$, then $T_{1,y}(M)=\chi_y(M)$ is the Hirzebruch $\chi_y$-genus (also called the $T_y$-genus) of $M$, where the Hirzebruch $\chi_y$-genus is the genus belonging to the power series
\begin{center}
$\displaystyle \frac{u(e^{u(1+y)}+y)}{e^{u(1+y)}-1}$.
\end{center}
Multiplying both of the numerator and denominator by $e^{-u(1+y)}$, this is equivalent to
\begin{center}
$\displaystyle \frac{u(1+ye^{-u(1+y)})}{1-e^{-u(1+y)}}$.
\end{center}
As for the $T_{x,y}$-genus, write $\chi_y(M)=\sum_{i=0}^n \chi^i(M) \cdot y^i$ for some integers $\chi^i(M)$; then $\chi^i(M)=T^i$, $0 \leq i \leq n$.

\begin{theorem} \label{t21}
Let the circle act on a $2n$-dimensional compact unitary manifold $M$ with a discrete fixed point set. For each integer $i$ such that $0 \leq i \leq n$,
\begin{center}
$\displaystyle \chi^i(M)=\sum_{p \in M^{S^1}} \epsilon(p) \cdot \frac{\sigma_i (t^{w_{p,1}}, \cdots, t^{w_{p,n}})}{\prod_{j=1}^n (1-t^{w_{p,j}})} = (-1)^i \cdot N_i = (-1)^i \cdot N_{n-i}=(-1)^n \cdot \chi^{n-i}(M)$,
\end{center}
where $\chi_y(M)=\sum_{i=0}^n \chi^i(M) \cdot y^i$ is the Hirzebruch $\chi_y$-genus of $M$, $t$ is an indeterminate, $\sigma_i$ is the $i$-th elementary symmetric polynomial in $n$ variables, and
\begin{center} 
$N_i=|\{p \in M^{S^1}\,|\,n_p=i, \epsilon(p)=+1\}|-|\{p \in M^{S^1}\,|\,n_p=i, \epsilon(p)=-1\}|$ 
\end{center}
is the number of fixed points $p$ with $n_p=i$ and $\epsilon(p)=+1$ minus the number of fixed points $p$ with $n_p=i$ and $\epsilon(p)=-1$.
\end{theorem}

\begin{proof} By Theorem \ref{t20}, the $T_{x,y}$-genus of $M$ is
\begin{center}
$\displaystyle T_{x,y}(M)=\sum_{p \in M^{S^1}} \epsilon(p) \cdot \prod_{i=1}^n \frac{x+y \cdot t^{w_{p,i}}}{1-t^{w_{p,i}}}$
\end{center}
for all indeterminates $t$. Evaluating the $T_{x,y}$-genus of $M$ at $x=1$ and collecting the $y^i$-terms, it follows that
\begin{equation} \label{e0}
\displaystyle \chi^i(M)=\sum_{p \in M^{S^1}} \epsilon(p) \cdot \frac{\sigma_i (t^{w_{p,1}}, \cdots, t^{w_{p,n}})}{\prod_{j=1}^n (1-t^{w_{p,j}})}
\end{equation}
for all indeterminates $t$, for $0 \leq i \leq n$. This proves the first equation.

For the second equation, fix $i$, $0 \leq i \leq n$. In Equation \eqref{e0} we shall do the following: if $a$ is a negative integer, then we manipulate 
\begin{center}
$\displaystyle \frac{1}{1-t^a}=\frac{t^{-a}}{(1-t^a) \cdot t^{-a}}=\frac{t^{-a}}{t^{-a}-1}=-\frac{t^{-a}}{1-t^{|a|}}$.
\end{center}
Rewriting Equation \eqref{e0},
\begin{center}
$\displaystyle \chi^i(M)=\sum_{p \in M^{S^1}} \epsilon(p) \cdot \frac{\sigma_i (t^{w_{p,1}}, \cdots, t^{w_{p,n}})}{\prod_{j=1}^n (1-t^{w_{p,j}})}=\sum_{p \in M^{S^1}} \epsilon(p) \cdot (-1)^{n_p} \frac {[\prod_{w_{p,k}<0} t^{-w_{p,k}}] \cdot \sigma_i (t^{w_{p,1}}, \cdots, t^{w_{p,n}})}{\prod_{j=1}^{n} (1-t^{|w_{p,j}|})}=\sum_{p \in M^{S^1}} \epsilon(p) \cdot (-1)^{n_p} [\prod_{w_{p,k}<0} t^{-w_{p,k}}] \cdot \sigma_i (t^{w_{p,1}}, \cdots, t^{w_{p,n}}) \prod_{j=1}^{n} (1+ t^{|w_{p,j}|} +t^{2|w_{p,j}|} + \cdots)$
\end{center}

Let $p$ be a fixed point of index $i$. Then in the last expression,
\begin{center}
$\epsilon(p) (-1)^{n_p} [\prod_{w_{p,k}<0} t^{-w_{p,k}}] \sigma_i (t^{w_{p,1}}, \cdots, t^{w_{p,n}}) \prod_{j=1}^{n} (1+ t^{|w_{p,j}|} +t^{2|w_{p,j}|} + \cdots)=\epsilon(p) \cdot (-1)^i (1+ f_p(t))$,
\end{center}
where $f_p(t)$ is a formal power series without a constant term.

Let $p$ be a fixed point such that $n_p \neq i$. Then in the last expression,
\begin{center}
$\epsilon(p) (-1)^{n_p} [\prod_{w_{p,k}<0} t^{-w_{p,k}}] \sigma_i (t^{w_{p,1}}, \cdots, t^{w_{p,n}}) \prod_{j=1}^{n} (1+ t^{|w_{p,j}|} +t^{2|w_{p,j}|} + \cdots)$
\end{center}
does not have any constant term. Therefore, taking $t=0$ in Equation \eqref{e0}, it follows that
\begin{center}
$\displaystyle \chi^i(M)=\sum_{p \in M^{S^1}} \epsilon(p) \cdot \frac{\sigma_i (t^{w_{p,1}}, \cdots, t^{w_{p,n}})}{\prod_{j=1}^n (1-t^{w_{p,j}})} = (-1)^i \cdot N_i$.
\end{center}
This proves the second equation.

For the third equation, we reverse the circle action on $M$; let $S^1$ act on $M$ by $g^{-1} \cdot m$ for all $g \in S^1 \subset \mathbb{C}$ and $m \in M$, where $g \cdot m$ denotes the original circle action on $M$. Reversing the circle action preserves the unitary structure, and reverses the sign of every weight at any fixed point. If we apply the first and second equations to the reversed circle action, it follows that 
\begin{center}
$\displaystyle \chi^i(M)=\sum_{p \in M^{S^1}} \epsilon(p) \cdot \frac{\sigma_i (t^{-w_{p,1}}, \cdots, t^{-w_{p,n}})}{\prod_{j=1}^n (1-t^{-w_{p,j}})} = (-1)^i \cdot N_{n-i}$,
\end{center}
for $0 \leq i \leq n$. Therefore, it follows that $N_i=N_{n-i}$ for all $0 \leq i \leq n$. This proves the third equation. Since $\chi^{n-i}(M)=(-1)^{n-i} \cdot N_{n-i}$, the last equation follows. \end{proof}

\begin{rem}
Let $M$ be a compact almost complex manifold. Suppose that the circle acts on $M$ with a non-empty discrete fixed point set. Then $\epsilon(p)=+1$ for every fixed point $p \in M^{S^1}$ if we choose a natural orientation coming from the almost complex structure. Therefore, Theorem \ref{t21} recovers that the number of fixed points of index $i$ is equal to the number of fixed points of index $n-i$ for any $0 \leq i \leq n$, where $\dim M=2n$. Moreover, there exists $0 \leq i \leq n-1$ such that both of $\chi^i(M)$ and $\chi^{i+1}(M)$ are not zero; see \cite{Ko1}, \cite{JT}, \cite{J1}. However, these two properties need not hold for a unitary manifold. For instance, as a unitary manifold, $S^{2n}$ admits a circle action with two fixed points $p$ and $q$ that have the same multisets of weights and $\epsilon(p)=-\epsilon(q)$; see Example \ref{ex3}. Therefore, by Theorem \ref{t21} it follows that $\chi_y(S^{2n})=0$ for this action. On the other hand, as a complex manifold, $S^2$ admits a circle action with 2 fixed points, one of index 0 and the other of index 1. The Hirzebruch $\chi_y$-genus of $S^2$ with the standard complex structure is $\chi_y(S^2)=1-y$. Therefore, the Hirzebruch $\chi_y$-genus of a manifold depends on unitary structures on it.
\end{rem}

More generally, let the circle act on a compact unitary manifold $M$ with fixed points, not necessarily discrete. Let $F$ be a fixed component. By Lemma \ref{l23} below, $F$ is naturally a unitary manifold, and for $p \in F$, the normal space $N_pF$ of $F$ to $M$ at $p$ is a complex vector space that decomposes into the sum of complex 1-dimensional vector spaces
\begin{center}
$N_pF=L_{F,1} \oplus \cdots \oplus L_{F,m},$
\end{center}
and the circle acts on each $L_{F,i}$ by multiplication by $g^{w_{F,i}}$ for some non-zero integer $w_{F,i}$ for all $g \in S^1 \subset \mathbb{C}$, $1 \leq i \leq m$, where $2m$ is the codimension of $F$ in $M$. The integers $w_{F,1}$, $\cdots$, $w_{F,m}$ are the same for all $p \in F$, and called the \textbf{weights} of $F$. Let $d(-,F)$ and $d(+,F)$ be the numbers of negative weights and positive weights in the normal bundle $NF$ of $F$, respectively. The Hirzebruch $\chi_y$-genus of $M$ satisfies the following formula, called the \textbf{Kosniowski formula}:
\begin{center}
$\displaystyle \chi_y(M)=\sum_{F \subset M^{S^1}} (-y)^{d(-,F)} \chi_y(F)=\sum_{F \subset M^{S^1}} (-y)^{d(+,F)} \chi_y(F)$.
\end{center}
The formula was first established for a complex manifold by Kosniowski in \cite{Ko}, and was established for a unitary manifold by Hattori and Taniguchi in \cite{HT}, and later by Kosniowski and Yahia \cite{KY}.

Let an abelian Lie group $G$ act properly on a stable complex manifold $M$; let $J$ be a given complex structure on $TM \oplus \underline{\mathbb{R}}^k$ that $G$ preserves. Let $H$ be a closed subgroup of $G$; then $H$ also acts on $M$. Let $F$ be a connected component of the set $M^H$ of points in $M$ that are fixed by the $H$-action. The $G$-action on $M$ restricts to act on $F$, and the sub-bundle $TF \oplus \underline{\mathbb{R}}^k|_F$ consists of the vectors that are fixed by $H$, and is $G$-equivariant and complex. The normal bundle $NF$ of $F$ is also a complex bundle since it is isomorphic to $(TM \oplus \underline{\mathbb{R}}^k)|_F / (TF \oplus \underline{\mathbb{R}}^k|_F)$.

\begin{lem} \cite{GGK} \label{l23} Let an abelian Lie group $G$ act properly on a stable complex manifold $M$ preserving a given complex structure on $TM \oplus \underline{\mathbb{R}}^k$ for some non-negative integer $k$. Let $H$ be a closed subgroup of $G$. Let $F$ be a connected component of $M^H$. An equivariant stable complex structure on $M$ induces an equivariant stable complex structure on $F$ and a fiberwise complex structure on $NF$ such that $G$ acts on $NF$ by complex bundle automorphisms.
\end{lem}

Let the circle act on a manifold $M$. The \textbf{equivariant cohomology} of $M$, denoted by $H_{S^1}^\ast(M)$, is
\begin{center}
$\displaystyle H_{S^1}^\ast(M)=H^\ast(M \times_{S^1} ES^1)$,
\end{center}
where $ES^1$ is a contractible space on which the circle acts freely. Suppose that $M$ is compact and oriented. The projection map $\pi:M \times_{S^1} ES^1 \to BS^1$ induces a push-forward map
\begin{center}
$\displaystyle \pi_\ast:H_{S^1}^i(M;\mathbb{Z}) \to H^{i -\dim M}(BS^1;\mathbb{Z})$
\end{center}
for all $i \in \mathbb{Z}$, where $BS^1$ is the classifying space of the circle. We also denote the projection map $\pi_\ast$  by $\int_M$, and the map $\int_M$ is given by integration over the fiber $M$.

\begin{theo} \emph{[The Atiyah-Bott-Berline-Vergne localization theorem]} \cite{AB} \label{t25}
Let the circle $S^1$ act on a compact oriented manifold $M$. Let $ \alpha \in H_{S^1}^{\ast}(M;\mathbb{Q})$. As an element of $\mathbb{Q}(t)$,
\begin{center}
$\displaystyle \int_{M} \alpha = \sum_{F \subset M^{S^1}} \int_{F} \frac{\alpha|_{F}}{e_{S^1}(NF)}$,
\end{center}
where the sum is taken over all fixed components, and $e_{S^1}(NF)$ is the equivariant Euler class of the normal bundle to $F$ in $M$.
\end{theo}

\section{Properties of unitary $S^1$-manifolds} \label{s3}

In this section, we prove various results for circle actions on compact unitary manifolds with discrete fixed points sets, generalizing results for almost complex manifolds. For a compact unitary $S^1$-manifold with a discrete fixed point set, the following proposition proves a property that the minimum of the absolute values of weights satisfies, and is a generalization of the result for almost complex $S^1$-manifolds in \cite{JT}, \cite{J2} to unitary $S^1$-manifolds; if $\epsilon(p)=+1$ for all $p \in M^{S^1}$, the proposition below recovers the result for almost complex manifolds. It will play an important role for proving a relationship between weights at the fixed points, Lemma \ref{l33}. The idea of proof is due to Kosniowski \cite{Ko}, while here we must take into account of the sign at each fixed point.

\begin{pro} \label{p31} 
Let the circle act on a $2n$-dimensional compact unitary manifold $M$ with a discrete fixed point set. Let 
\begin{center}
$a=\min\{|w_{p,j}| ; 1 \leq j \leq n, p \in M^{S^1}\}$
\end{center}
be the minimum of the absolute values of weights that occur at any fixed point. For any $i \in \{0,1,\cdots,n-1\}$,
\begin{center}
$\displaystyle \sum_{p \in M^{S^1}, n_p=i} \epsilon(p) \cdot N_p(a)=\sum_{p \in M^{S^1}, n_p=i+1} \epsilon(p) \cdot N_p(-a)$,
\end{center}
where $N_p(w)=|\{j \, | \, w_{p,j}=w, 1 \leq j \leq n\}|$ is the number of times weight $w$ occurs at $p$, for any integer $w$. \end{pro}

\begin{proof} We first show that
\begin{equation} \label{e1}
\displaystyle \sum_{n_p=i} \epsilon(p) \cdot [N_{p}(a)+N_{p}(-a)]=\sum_{n_p=i-1} \epsilon(p) \cdot N_{p}(a)+\sum_{n_p=i+1} \epsilon(p) \cdot N_{p}(-a),
\end{equation}
for $0 \leq i \leq n$. By Theorem \ref{t21}, for any $0 \leq i \leq n$,
\begin{center}
$\displaystyle \chi^i(M)=\sum_{p \in M^{S^1}} \epsilon(p) \cdot \frac{\sigma_i (t^{w_{p,1}}, \cdots, t^{w_{p,n}})}{\prod_{j=1}^n (1-t^{w_{p,j}})}$
\end{center}
for all indeterminates $t$, where $\sigma_i$ is the $i$-th elementary symmetric polynomial in $n$ variables. We manipulate the equation above to the following equation
\begin{equation} \label{e2}
\displaystyle \chi^i(M)=\sum_{p \in M^{S^1}} \epsilon(p) \cdot (-1)^{n_p} \frac {[\prod_{w_{p,k}<0} t^{-w_{p,k}}] \cdot \sigma_i (t^{w_{p,1}}, \cdots, t^{w_{p,n}})}{\prod_{j=1}^{n} (1-t^{|w_{p,j}|})}.
\end{equation}

For each fixed point $p$, let 
\begin{center}
$\displaystyle J_p=[\prod_{w_{p,k}<0} t^{-w_{p,k}}] \cdot \sigma_i (t^{w_{p,1}}, \cdots, t^{w_{p,n}})$ and $\displaystyle K_p=\prod_{j=1}^{n} (1-t^{|w_{p,j}|})$.
\end{center}
Also, let $\displaystyle A=\sum_{p \in M^{S^1}} (N_p(a)+N_p(-a))$ and $\displaystyle B=\prod_{p \in M^{S^1}} K_p$. Then we can rewrite Equation \eqref{e2} as
\begin{center}
$\displaystyle \chi^i(M)=\sum_{p \in M^{S^1}} \epsilon(p) \cdot (-1)^{n_p} \frac{J_p}{K_p}$
\end{center}
Moreover,
\begin{center}
$\displaystyle B=\prod_{p \in M^{S^1}} K_p=1-At^a+g_1(t)$, 
\end{center}
where $g_1(t)$ is a formal power series that does not have a constant term and $t^a$-term.

Let $p$ be a fixed point of index $i$. Then $J_p=1+f_p(t)$, where $f_p(t)$ is a formal power series that does not have a constant term and $t^a$-term. Therefore, 
\begin{center}
$J_p \cdot B / K_p=1-(A-N_p(a)-N_p(-a))t^a+g_p(t)$, 
\end{center}
where $g_p(t)$ is a formal power series without a constant term and $t^a$-term.

Let $p$ be a fixed point of index $i+1$. Then $J_p=N_{p}(- a) \cdot t^a+f_p(t)$, where $f_p(t)$ is a formal power series that does not have a constant term and $t^a$-term. Therefore, $J_p \cdot B/K_p=N_p (- a)t^a+g_p(t)$, where $g_p(t)$ is a formal power series without a constant term and $t^a$-term.

Let $p$ be a fixed point of index $i-1$. Then $J_p=N_{p}(a) \cdot t^a+f_p(t)$, where $f_p(t)$ is a formal power series that does not have a constant term and $t^a$-term. Therefore, $J_p \cdot B/K_p=N_p (a)t^a+g_p(t)$, where $g_p(t)$ is a formal power series without a constant term and $t^a$-term.

Let $p$ be a fixed point of index $j$ such that $j \neq i, i \pm 1$. Then $J_p=f_p(t)$, where $f_p(t)$ is a formal power series that does not have a constant term and $t^a$-term. Therefore, $J_p \cdot B/K_p=g_p(t)$, where $g_p(t)$ is a formal power series without a constant term and $t^a$-term.

We multiply Equation \eqref{e2} by $B$ to have
\begin{center}
$\displaystyle \chi^i(M)[1-At^a+g_1(t)]=(-1)^i \sum_{n_p=i} \epsilon(p) + \{(-1)^i \sum_{n_p=i} \epsilon(p) \cdot (N_p(a)+N_p(-a)-A)+ (-1)^{i-1} \sum_{n_p=i-1} \epsilon(p) \cdot  N_p(a)+(-1)^{i+1} \sum_{n_p=i+1} \epsilon(p) \cdot N_p(-a)\}t^a + g_2(t)$,
\end{center}
where $g_2(t)$ is a formal power series that does not have a constant term and $t^a$-term. Since constant terms match, $\chi^i(M)=(-1)^i \cdot \sum_{n_p=i} \epsilon(p)$. With this, comparing the coefficients of $t^a$-terms in the equation above, Equation \eqref{e1} holds.

Applying Equation \eqref{e1} for $i=0$, we have
\begin{center}
$\displaystyle \sum_{p \in M^{S^1}, n_p=0} \epsilon(p) \cdot N_p(a)=\sum_{p \in M^{S^1}, n_p=1} \epsilon(p) \cdot N_p(-a)$.
\end{center}

Next, applying \eqref{e1} for $i=1$, we have
\begin{center}
$\displaystyle \sum_{n_p=1} \epsilon(p) \cdot [N_{p}(a)+N_{p}(-a)]=\sum_{n_p=0} \epsilon(p) \cdot N_{p}(a)+\sum_{n_p=2} \epsilon(p) \cdot N_{p}(-a).$
\end{center}
Since $\displaystyle \sum_{p \in M^{S^1}, n_p=0} \epsilon(p) \cdot N_p(a)=\sum_{p \in M^{S^1}, n_p=1} \epsilon(p) \cdot N_p(-a)$, it follows that
\begin{center}
$\displaystyle \sum_{p \in M^{S^1}, n_p=1} \epsilon(p) \cdot N_p(a)=\sum_{p \in M^{S^1}, n_p=2} \epsilon(p) \cdot N_p(-a)$.
\end{center}
Continuing this, the proposition follows. \end{proof}

Let $M$ be a compact unitary $S^1$-manifold. For a positive integer $w$, consider an action of the group $\mathbb{Z}_w$ on $M$, where $\mathbb{Z}_w$ acts on $M$ as a subgroup of $S^1$. Then the set $M^{\mathbb{Z}_w}$ of points in $M$ that are fixed by the $\mathbb{Z}_w$-action is also a compact unitary $S^1$-manifold.

\begin{lem} \label{l32}
Let the circle act on a compact unitary manifold $M$ with a discrete fixed point set; let $J$ be a complex structure on $TM \oplus \underline{\mathbb{R}}^k$ for some non-negative integer $k$. Let $w$ be a positive integer. Let $F$ be a connected component of $M^{\mathbb{Z}_w}$, where $\mathbb{Z}_w$ acts on $M$ as a subgroup of $S^1$. Then the circle action on $M$ restricts to act on $F$ so that $F$ is naturally a unitary $S^1$-manifold; there is a complex structure $J_1$ on $TF \oplus \underline{\mathbb{R}}^k|_F$ preserved by the circle action. Moreover, $\epsilon_M(p)=\epsilon_F(p)$ for any $p \in F \cap M^{S^1}$.
\end{lem}

\begin{proof}
By Lemma \ref{l23}, there are natural complex structures $J_1$ and $J_2$ on $TF \oplus \underline{\mathbb{R}}^k|_F$ and $NF$, so that $G$ acts on $F$ preserving the complex structure $J_1$ on $TF \oplus \underline{\mathbb{R}}^k|_F$, and $G$ acts on $NF$ preserving $J_2$. This proves the first claim. Moreover, the unitary structure on $F$ gives an orientation of $F$ and the complex structure $J_2$ gives an orientation on $NF$; these orientations together agree with the orientation of $M$.

Let $p \in F \cap M^{S^1}$. Since $(TM \oplus \underline{\mathbb{R}}^k)|_F = TF \oplus NF \oplus \underline{\mathbb{R}}^k|_F$, the agreement $\epsilon_M(p)$ of the orientation on $T_pM$ coming from the orientation of $M$ and the orientation on $T_pM$ coming from the complex structure on $T_pM$ is the same as the agreement $\epsilon_F(p)$ of the orientation on $T_pF$ coming from the orientation of $F$ and the orientation on $T_pF$ coming from the complex structure on $T_pF$. \end{proof}

For a compact almost complex $S^1$-manifold with a discrete fixed point set, in \cite{H} Hattori proved that for any integer $w$, the number of times weight $w$ occurs over all fixed points, counted with multiplicity, is equal to the number of times weight $-w$ occurs over all fixed points, counted with multiplicity. Moreover, Hattori proved that for a compact unitary $S^1$-manifold with a discrete fixed point set, for any positive integer $w$,
\begin{equation} \label{eq4}
\displaystyle \sum_{p \in M^{S^1}} (N_p(w)+N_p(-w)) \equiv 0 \mod 2
\end{equation}
where $N_p(w)=|\{j \, | \, w_{p,j}=w, 1 \leq j \leq n\}|$. Applying Proposition \ref{p31}, we prove the following lemma, which proves a stronger result for unitary $S^1$-manifolds than
Equation \eqref{eq4} and which generalizes the result for almost complex manifolds by Hattori \cite{H} mentioned above. For a unitary $S^1$-manifold, when a weight $w$ occurs, $-w$ need not occur as a weight; see Example \ref{ex3}.

\begin{lem} \label{l33}
Let the circle act on a compact unitary manifold $M$ with a discrete fixed point set. For any integer $w$,
\begin{center}
$\displaystyle \sum_{p \in M^{S^{1}}} \epsilon(p)  \cdot N_{p}(w)=\sum_{p \in M^{S^{1}}} \epsilon(p) \cdot N_{p}(-w)$,
\end{center}
where $N_{p}(w)=|\{j \, | \, w_{p,j}=w, 1 \leq j \leq n\}|$ is the number of times $w$ occurs as weights at $p$. \end{lem}

\begin{proof}
Let $w$ be a positive integer. The group $\mathbb{Z}_w$ acts on $M$ as a subgroup of $S^1$. Let $F$ be a connected component of $M^{\mathbb{Z}_w}$ such that $F \cap M^{S^1} \neq \emptyset$. By Lemma \ref{l32}, the circle action on $M$ restricts to act on $F$ so that $F$ is naturally a unitary $S^1$-manifold; there is a complex structure $J_1$ on $TF \oplus \underline{\mathbb{R}}^k|_F$ preserved by the circle action. Moreover, $\epsilon_M(p)=\epsilon_F(p)$ for all $p \in F \cap M^{S^1}$. Let $p$ be a fixed point of the circle action on $F$. Then the weights at $p$ in $T_pF$ are multiples of $w$, and hence $w$ is the smallest positive weight for the circle action on $F$. Let $\dim F=2m$. For any $0 \leq i \leq m-1$, applying Proposition \ref{p31} to the circle action on $F$, 
\begin{center}
$\displaystyle \sum_{p \in F^{S^1}, n_p^F=i} \epsilon_F(p) \cdot N_p^F(w)=\sum_{p \in F^{S^1}, n_p^F=i+1} \epsilon_F(p) \cdot N_p^F(-w)$.
\end{center}
Here, $N_{p}^F(w)=|\{j \, | \, w_{p,j}=w, w_{p,j}\textrm{ is a weight in }T_pF\}|$ is the number of times $w$ occurs in $T_pF$ and $n_p^F$ is the number of negative weights in $T_pF$ for all $p \in F^C=F \cap M^{S^1}$. Summing over all $0 \leq i \leq m-1$, it follows that
\begin{center}
$\displaystyle \sum_{p \in F^{S^1}} \epsilon_F(p) \cdot N_p^F(w)=\sum_{p \in F^{S^1}} \epsilon_F(p) \cdot N_p^F(-w)$.
\end{center}
For any $p \in F^{S^1}$, a weight in $T_pM$ is in $T_pF$ if and only if it is divisible by $w$ and hence $N_p^F(w)=N_p(w)$ and $N_p^F(-w)=N_p(-w)$. With that $\epsilon_M(p)=\epsilon_F(p)$ for all $p \in F \cap M^{S^1}=F^{S^1}$, we have
\begin{center}
$\displaystyle \sum_{p \in F^{S^1}} \epsilon_M(p) \cdot N_p(w)=\sum_{p \in F^{S^1}} \epsilon_M(p) \cdot N_p(-w)$.
\end{center}
Since this holds for any positive integer $w$ and any connected component of $M^{\mathbb{Z}_w}$, the lemma holds. \end{proof}

By Lemma \ref{l33}, it follows that for a compact unitary $S^1$-manifold with a discrete fixed point set, the total number of weights over all fixed points must be even. Thus, we obtain the following corollary; Pelayo and Tolman proved an analogous result for symplectic $S^1$-actions on symplectic manifolds \cite{PT}.

\begin{cor} \label{c34} 
Let the circle act on a $2n$-dimensional compact unitary manifold with $k$ fixed points. If $k$ is odd, then $n$ is even.
\end{cor}

For a symplectic $S^1$-action on a symplectic manifold $M$, Tolman \cite{T} proved a relationship between weights at fixed points that lie in the same connected component of $M^{\mathbb{Z}_w}$; if $p$ and $p'$ are in the same component of $M^{\mathbb{Z}_w}$, then the weights at $p$ and at $p'$ are equal modulo $w$. The proof naturally extends to almost complex $S^1$-manifolds, and Godinho and Sabatini stated without a proof in \cite{GS}. The same argument holds for unitary $S^1$-manifolds.

\begin{lem} \label{l35} 
Let the circle act on a compact unitary manifold $M$. Let $p$ and $p'$ be fixed points that lie in the same component $F$ of $M^{\mathbb{Z}_w}$, for some positive integer $w$. Then the $S^{1}$-weights at $p$ and at $p'$ are equal modulo $w$.
\end{lem}

Lemma \ref{l35} means that if $p$ and $p'$ are in the same component of $M^{\mathbb{Z}_w}$, there exists a bijection $\pi:\{1,\cdots,n\} \to \{1,\cdots, n\}$ so that $w_{p,i} \equiv w_{p', \pi(i)} \mod w$ for $i \in \{1,\cdots,n\}$. We can check this lemma with examples, for instance Examples \ref{ex1}, \ref{ex2}, and \ref{ex3} below.

Consider a circle action on a compact unitary manifold $M$ with a discrete fixed point set. The \textbf{Chern class map} is a map $c_1(M) : M^{S^1} \to \mathbb{Z}$ that assigns each fixed point $p$ the first equivariant Chern class $c_1(M)(p)$ at $p$, ignoring a generator $x$ of $H^*(\mathbb{CP}^\infty)$. For any fixed point $p \in M^{S^1}$, $c_1(M)(p)$ is equal to the sum of the weights at $p$ (times $x$). In \cite{H}, Hattori also proved that for a compact almost complex $S^1$-manifold with a discrete fixed point set, the sum of the first equivariant Chern class at the fixed points vanishes. That is,
\begin{center}
$\displaystyle \sum_{p \in M^{S^1}} c_1(M)(p)=0$.
\end{center}
We generalize this to compact unitary $S^1$-manifolds as follows.

\begin{cor} \label{c36} 
If the circle acts on a compact unitary manifold $M$ with a discrete fixed point set, then 
\begin{center}
$\displaystyle \sum_{p \in M^{S^1}} \epsilon(p) \cdot c_1(M)(p)=0$,
\end{center}
where $c_1(M)(p)$ is the first equivariant Chern class of $M$ at $p$. \end{cor}

\begin{proof}
For each integer $w$, by Lemma \ref{l33},
\begin{center}
$\displaystyle \sum_{p \in M^{S^1}} \epsilon(p) \cdot N_p(w) \cdot w + \sum_{p \in M^{S^1}} \epsilon(p) \cdot N_p(-w)\cdot (-w)=0$,
\end{center}
where $N_p(w)=|\{j \, | \, w_{p,j}=w, 1 \leq j \leq n\}|$ is the number of times weight $w$ occurs at $p$. Therefore,
\begin{center}
$\displaystyle \sum_{p \in M^{S^1}} \epsilon(p) \cdot c_1(M)(p)=\sum_{p \in M^{S^1}} \epsilon(p) (w_{p,1}+\cdots+w_{p,n})=\sum_{p \in M^{S^1}} \epsilon(p) \cdot [\sum_{w \in \mathbb{N}} (N_p(w) \cdot w + N_p(-w) \cdot (-w))]=\sum_{w \in \mathbb{N}} [\sum_{p \in M^{S^1}} \epsilon(p) \cdot N_p(w) \cdot w + \sum_{p \in M^{S^1}} \epsilon(p) \cdot N_p(-w) \cdot (-w)]=\sum_{w \in \mathbb{N}}0=0$.
\end{center}
This proves the corollary. \end{proof}

In \cite{PT}, Pelayo and Tolman derived some result for a symplectic $S^1$-manifold under an assumption on the cardinality of the set $\{c_1(M)(p) \, | \, p \in M^{S^1}\}$. The result naturally extends to almost complex $S^1$-manifolds, and the author announced it in \cite{J1} without a proof. The result also extends to a unitary manifold, if in the proof of \cite{PT} we take $\epsilon(p)$ into account, for all $p \in M^{S^1}$.

\begin{lem} \label{l37}
Let the circle act on a $2n$-dimensional compact unitary manifold $M$ with a discrete fixed point set. If the set $\{c_1(M)(p) \, | \, p \in M^{S^1}\}$ contains at most $n$ elements, then
\begin{center}
$\displaystyle \sum_{p \in M^{S^1}, \, c_1(M)(p)=k} \epsilon(p) \cdot \frac{1}{\prod_{i=1}^n w_{p,i}}=0$ for all $k \in \mathbb{Z}$.
\end{center}
\end{lem}

\begin{proof}
Let $\{c_1(M)(p) \, | \, p \in M^{S^1}\}=\{k_1,\cdots,k_l\} \subset \mathbb{Z}$. For each $1 \leq i \leq l$, define 
\begin{center}
$\displaystyle A_i:=\sum_{p \in M^{S^1},c_1(M)(p)=k_i}\epsilon(p) \cdot \frac{1}{\prod_{m=1}^n w_{p,m}}$.
\end{center}
Let $B$ be an $l \times l$ matrix $B_{i,j}=(k_i)^{j-1}$ for $1 \leq i \leq l$ and $1 \leq j \leq l$. Since $l \leq n$, for $0 \leq j<l$, the number $2j - \dim M$ is negative. Thus, by a dimensional reason that $\int_M (c_1(M))^j$ is a map from $H_{S^1}^{2j}(M)$ to $H_{S^1}^{2j- \dim M}(BS^1)$, which is trivial, the image of $(c_1(M))^j$ under the map $\int_M$ vanishes;
\begin{center}
$\displaystyle \int_M (c_1(M))^j=0$.
\end{center}
On the other hand, by Theorem \ref{t25}, we have
\begin{center}
$\displaystyle \int_M (c_1(M))^j=\sum_{p \in M^{S^1}} \epsilon(p) \cdot \frac{(c_1(M)(p))^j}{\prod_{m=1}^n w_{p,m}}=\sum_{1 \leq i \leq l}\sum_{p \in M^{S^1},c_1(M)(p)=k_i} \epsilon(p) \cdot \frac{(k_i)^j}{\prod_{m=1}^n w_{p,m}}=\sum_{1 \leq i \leq l} B_{i,j+1} \cdot A_i$.
\end{center}
Therefore, we have that $B \cdot (A_1, \cdots, A_l)=(0,\cdots,0)$. Since $B$ is a Vandermonde matrix, $\det(B) \neq 0$, and this implies that $A_1=\cdots=A_l=0$.
\end{proof}

We say that a map $f:X \to Y$ is somewhere injective, if there is a point $y \in Y$ such that $f^{-1}(\{y\})$ is the singleton. As for symplectic manifolds and almost complex manifolds, Lemma \ref{l37} has the following consequence.

\begin{theorem} \label{t38} 
Let the circle act on a compact unitary manifold $M$. If the Chern class map is somewhere injective, then the action has at least $\frac{1}{2}\dim M+1$ fixed points.
\end{theorem}

\begin{proof}
By assumption, there exists $k \in \mathbb{Z}$ such that
\begin{center}
$\displaystyle \sum_{p \in M^{S^1}, \, c_1(M)(p)=k} \epsilon(p) \cdot \frac{1}{\prod_{i=1}^n w_{p,i}} \neq 0$.
\end{center}
By Lemma \ref{l37}, this implies that the set $\{c_1(M)(p) \, | \, p \in M^{S^1}\}$ contains at least $\frac{1}{2}\dim M+1$ elements, and hence there are at least $\frac{1}{2}\dim M+1$ fixed points. \end{proof}

In \cite{TW}, Tolman and Weitsman proved that if the circle acts semi-freely on a $2n$-dimensional compact almost complex manifold $M$ with a non-empty discrete fixed point set, then $\textrm{Todd}(M)>0$ and there are $\textrm{Todd}(M) \cdot 2^n$ fixed points. In \cite{J3}, the author generalized the result to so called primitive circle actions. We generalize the result for semi-free circle actions on almost complex manifolds to semi-free circle actions on unitary manifolds. While for an almost complex manifold we get a precise number of fixed points in terms of the Todd genus of the manifold, for a unitary manifold we get a lower bound on the number fixed points.

\begin{theo} \label{t39}
Let the circle act semi-freely on a $2n$-dimensional compact unitary manifold $M$ with a discrete fixed point set. Then the Hirzebruch $\chi_y$-genus of $M$ is $\chi_y(M)=\textrm{Todd}(M) \cdot (1-y)^n$. Moreover, $N_i = \textrm{Todd}(M) \cdot {n \choose i}$ for any $0 \leq i \leq n$, where $N_i=|\{p \in M^{S^1} ; \epsilon(p)=+1,n_p=i\}|-|\{p \in M^{S^1} ; \epsilon(p)=-1,n_p=i\}|$. In particular, if $\textrm{Todd}(M) \neq 0$, there are at least $|\textrm{Todd}(M)| \cdot 2^n$ fixed points.
\end{theo}

\begin{proof}
Since the action is semi-free, every weight at any fixed point is either $+1$ or $-1$. The Todd genus $\textrm{Todd}(M)$ of $M$ is equal to the Hirzebruch $\chi_y$-genus with $y=0$. Therefore, $\textrm{Todd}(M)=\chi_0(M)=\sum_{i=0}^n \chi^i(M) \cdot 0^i=\chi^0(M)$. By Theorem \ref{t21},
\begin{center}
$\displaystyle \textrm{Todd}(M)=\chi^0(M)=N_0=\sum_{p \in M^{S^1}} \epsilon(p) \cdot \frac{1}{\prod_{i=1}^n (1-t^{\pm 1})}=\sum_{p \in M^{S^1}} \epsilon(p) \cdot \frac{(-t)^{n_p}}{(1-t)^n}=\frac{\sum_{i=0}^n N_i (-t)^i}{(1-t)^n}$.
\end{center}
Since $\chi^0(M)$ is an integer, the right side of the equation above must equal that integer. It follows that $\sum_{i=0}^n N_i (-t)^i=N_0 (1-t)^n$, and hence $N_i=N_0 \cdot {n \choose i}=\textrm{Todd}(M) \cdot {n \choose i}$ for $0 \leq i \leq n$. Since the number $|M^{S^1}|$ of fixed points satisfies $|M^{S^1}| \geq \sum_{i=0}^n |N_i| = \sum_{i=0}^n |\textrm{Todd}(M)| \cdot {n \choose i}=|\textrm{Todd}(M)| \cdot 2^n$, there are at least $|\textrm{Todd}(M)| \cdot 2^n$ fixed points. \end{proof}

In Theorem \ref{t39}, in the case that $\textrm{Todd}(M)>0$ and $M$ has precisely $\textrm{Todd}(M) \cdot 2^n$ fixed points, there exists a $2n$-dimensional compact connected almost complex manifold $M$ equipped with a semi-free circle action \cite{J4}.

For a compact almost complex $S^1$-manifold with a discrete fixed point set, Godinho and Sabatini \cite{GS}, Tolman \cite{T}, and Tolman and the author \cite{JT} showed that there exists a multigraph that encodes the fixed point data of the manifold, which is a collection of multisets of weights at the fixed points. For an almost complex manifold, each edge of such a multigraph has a direction, and is labeled by a positive integer.

On the other hand, for a compact unitary $S^1$-manifold $M$ with a discrete fixed point set, to associate such a multigraph we need to resolve the sign issue that occurs at each fixed point; this can be done by introducing a signed (directed labeled) multigraph.

\begin{Definition} \label{d1}
\begin{enumerate}[(1)]
\item A \textbf{multigraph} is a pair $(V,E)$ where $V$ is a set of vertices and $E$ is a set of edges.
\item A multigraph $(V,E)$ is called \textbf{directed}, if there are maps $i:E \to V$ and $t:E \to V$ giving the initial vertex and terminal vertex of each edge.
\item A multigraph $(V,E)$ is called \textbf{labeled}, if there is a map $w$ from $E$ to the set of positive integers. For an edge $e$, $w(e)$ is called the \textbf{label} of $e$.
\item A multigraph is called \textbf{signed}, if every vertex has \textbf{sign} $+1$ or $-1$.
\end{enumerate}
\end{Definition}

\begin{Definition} \label{d2}
Let the circle act on a compact unitary manifold $M$ with a discrete fixed point set. We say that a directed labeled signed multigraph $\Gamma=(V,E)$ \textbf{describes} (the fixed point data of) $M$ if the following hold:
\begin{enumerate}[(1)]
\item The vertex set $V$ of $\Gamma$ is equal to the fixed point set $M^{S^1}$ of $M$.
\item For any vertex $p$, the sign of $p$ is equal to $\epsilon(p)$.
\item For any vertex $p$, the multiset of the weights at the corresponding fixed point $p$ are $\{\epsilon(p) \cdot w(e) \, | \, i(e)=p\} \sqcup \{- \epsilon(p) \cdot w(e) \, | \, t(e)=p\}$.
\item For every edge $e$, the fixed points $i(e)$ and $t(e)$ lie in the same connected component of $M^{\mathbb{Z}/(w(e))}$.
\end{enumerate}
\end{Definition}

Let $M$ be a compact almost complex manifold, equipped with a circle action having a discrete fixed point set. Then $\epsilon(p)=+1$ for every fixed point $p$. This means that if a directed labeled signed multigraph $\Gamma$ describes $M$, for any edge $e$ of $\Gamma$, a fixed point $p$ corresponding to the initial vertex $i(e)(=p)$ of $e$ has weight $\epsilon(p)\cdot w(e)=+w(e)$ and a fixed point $q$ corresponding to the terminal vertex $t(e)(=q)$ of $e$ has weight $\epsilon(q)\cdot (-w(e))=-w(e)$.

On the other hand, for a unitary manifold, the situation is more complicated. To discuss, suppose that a directed labeled signed multigraph $\Gamma$ describes a compact unitary $S^1$-manifold $M$ with a discrete fixed point set. Let $e$ be an edge of $\Gamma$, and let $w$ be the label of the edge $e$. Let $p$ be the initial vertex (fixed point) of $e$, and let $q$ be the terminal vertex (fixed point) of $e$. Then Definition \ref{d2} implies the following:
\begin{enumerate}[(1)]
\item Suppose that $\epsilon(p)=+1$ and $\epsilon(q)=+1$. Then $e$ gives $p$ weight $\epsilon(p) \cdot w(e)=+w$ and gives $q$ weight $-\epsilon(q) \cdot w(e)=-w$.
\item Suppose that $\epsilon(p)=+1$ and $\epsilon(q)=-1$. Then $e$ gives $p$ weight $\epsilon(p) \cdot w(e)=+w$ and gives $q$ weight $-\epsilon(q) \cdot w(e)=+w$. 
\item Suppose that $\epsilon(p)=-1$ and $\epsilon(q)=+1$. Then $e$ gives $p$ weight $\epsilon(p) \cdot w(e)=-w$ and gives $q$ weight $-\epsilon(q) \cdot w(e)=-w$. 
\item Suppose that $\epsilon(p)=-1$ and $\epsilon(q)=-1$. Then $e$ gives $p$ weight $\epsilon(p) \cdot w(e)=-w$ and gives $q$ weight $-\epsilon(q) \cdot w(e)=+w$. 
\end{enumerate}
In addition, the two fixed points $p$ and $q$ are in the same connected component $F$ of $M^{\mathbb{Z}/(w(e))}$, the set of points in $M$ that are fixed by the $\mathbb{Z}/(w(e))$-action on $M$; the group $\mathbb{Z}/(w(e))$ acts on $M$ as a subgroup of $S^1$. The component $F$ is naturally a unitary manifold by Lemma \ref{l23}, and is a smaller dimensional compact unitary submanifold if the label $w(e)$ of the edge is bigger than 1; if $w(e)=1$, then the subgroup $\mathbb{Z}/(w(e))$ is trivial and so $F=M^{\mathbb{Z}/(w(e))}=M$.

Generalizing the existence of a multigraph describing a compact almost complex $S^1$-manifold with a discrete fixed point set in \cite{GS}, \cite{JT}, \cite{J5}, \cite{T}, we prove that for a compact unitary $S^1$-manifold $M$ with a discrete fixed point set, there exists a directed labeled signed multigraph describing $M$. Any directed labeled multigraph describing a compact almost complex $S^1$-manifold with a discrete fixed point set is obtained from our multigraph describing a unitary $S^1$-manifold in this paper, simply by taking $\epsilon(p)=+1$ for all fixed points $p$. For a multigraph, by a self-loop we mean an edge whose initial vertex and terminal vertex coincide.

\begin{pro} \label{p312} Let the circle act on a compact unitary manifold $M$ with a discrete fixed point set. Then there exists a directed labeled signed multigraph describing $M$ that has no self-loops. \end{pro}

\begin{proof}
Assign a vertex to each fixed point; the sign of the vertex is equal to the sign of the corresponding fixed point. Let $w$ be a positive integer. Let $F$ be a connected component of $M^{\mathbb{Z}_w}$ such that $F \cap M^{S^1} \neq \emptyset$. By Lemma \ref{l32}, the circle action on $M$ restricts to act on $F$ so that $F$ is naturally a unitary $S^1$-manifold; there is a complex structure $J_1$ on $TF \oplus \underline{\mathbb{R}}^k$ preserved by the circle action; moreover, $\epsilon_M(p)=\epsilon_F(p)$ for all $p \in F \cap M^{S^1}$. Let $\dim F=2m$. For a fixed point $p \in F \cap M^{S^1}$, the smallest positive weight of the circle action in $T_pF$ is $w$. Therefore, by applying Proposition \ref{p31} to the circle action on $F$, for any $i \in \{0,1,\dots,m-1\}$,
\begin{equation} \label{eq5}
\displaystyle \sum_{p \in F^{S^1}, n_p^F=i} \epsilon_F(p) \cdot N_p^F(w)=\sum_{p \in F^{S^1}, n_p^F=i+1} \epsilon_F(p) \cdot N_p^F(-w),
\end{equation}
where $N_p^F(w)=|\{j  : w_{p,j}=w, 1 \leq j \leq n\}|$ is the number of times weight $w$ occurs in $T_pF$ at $p$, for any integer $w$. Here, $n_p^F$ is the number of negative weights in $T_pF$ for any $p \in F^{S^1}$. Therefore, for each $i \in \{0,1,\dots,m-1\}$ we can exhaust all weights $w$ over all fixed points of index $i$ in $F^{S^1}$ and all weights $-w$ over all fixed points of index $i+1$ in $F^{S^1}$ to draw edges with label $w$ as follows.
\begin{enumerate}[(a)]
\item If two fixed points $p$ and $q$ in $F^{S^1}$ satisfy $n_p^F=n_q^F=i$, $1=\epsilon_F(p)=-\epsilon_F(q)$, and both of $p$ and $q$ have weight $w$, then draw an edge from $p$ to $q$ giving label $w$. Both weights are taken from the left hand side of Equation \eqref{eq5}.
\item If two fixed points $p$ and $q$ in $F^{S^1}$ satisfy $n_p^F+1=n_q^F=i+1$, $1=\epsilon_F(p)=\epsilon_F(q)$, $p$ has weight $w$, and $q$ has weight $-w$, then draw an edge from $p$ to $q$ giving label $w$. One weight $w$ is taken from the left hand side and another weight $-w$ is taken from the right hand side of Equation \eqref{eq5}.
\item If two fixed points $p$ and $q$ in $F^{S^1}$ satisfy $n_p^F+1=n_q^F=i+1$, $-1=\epsilon_F(p)=\epsilon_F(q)$, $p$ has weight $w$, and $q$ has weight $-w$, then draw an edge from $q$ to $p$ giving label $w$. One weight $w$ is taken from the left hand side and another weight $-w$ is taken from the right hand side of Equation \eqref{eq5}.
\item If two fixed points $p$ and $q$ in $F^{S^1}$ satisfy $n_p^F=n_q^F=i+1$, $1=\epsilon_F(p)=-\epsilon_F(q)$, and both of $p$ and $q$ have weight $-w$, then draw an edge from $q$ to $p$ giving label $w$. Both weights are taken from the right hand side of Equation \eqref{eq5}.
\end{enumerate}

This recipe exhausts all weights $\pm w$ in Equation \eqref{eq5}. Since $\epsilon_M(p)=\epsilon_F(p)$ for all $p \in F \cap M^{S^1}$, by the way we draw edges, for any edge $e$, if $e$ has label $w$, the edge $e$ gives the initial vertex (fixed point) $r$ weight $\epsilon(r) \cdot w$ and gives the terminal vertex (fixed point) $s$ weight $-\epsilon(s) \cdot w$.

By repeating the above arguments for every positive integer $w$ and every connected component of $M^{\mathbb{Z}_w}$, this proposition holds. \end{proof}

Let $M$ be a compact unitary $S^1$-manifold with a discrete fixed point set. 
Let $w>1$ be an integer and let $p$ be a fixed point. Suppose that exactly $k$ weights $w_{p,i_1}$, $\cdots$, $w_{p,i_k}$ at $p$ are divisible by $w$. Let $F$ be a connected component of the set $M^{\mathbb{Z}_w}$ of points in $M$ fixed by the $\mathbb{Z}_w$-action, which contains $p$.
The circle action on $M$ restricts to act on $F$ and its fixed point set is $F^{S^1}=F \cap M^{S^1}$. For this action on $F$, the weights in $T_pF$ at $p$ are precisely those weights $w_{p,i_1}$, $\cdots$, $w_{p,i_k}$. This implies that the dimension of $F$ is $2k$, which is twice of the number of weights at $p$ that are multiples of $w$. Conversely, if $\dim F=2k$, then $p$ has exactly $k$ weights divisible by $w$. Let $\Gamma$ be a multigraph describing $M$. 
The above discussion means that for a vertex (fixed point) $p$, the number of edges of $p$ whose labels are divisible by $w$, is the half of the dimension of the isotropy submanifold $F$ fixed by the $\mathbb{Z}_w$-action, which is compact and unitary.

Moreover, a sub-multigraph of $\Gamma$ describes the $S^1$-action on the isotropy submanifold $F$. Let $\Gamma_F$ be a sub-multigraph of $\Gamma$ obtained as follows:
\begin{enumerate}
\item The vertex set of $\Gamma_F$ is equal to $F \cap M^{S^1}$.
\item For each vertex $q$ in $\Gamma_F$, edges of $q$ in $\Gamma_F$ are the edges of $q$ of $\Gamma$ whose labels are divisible by $w$.
\item For each vertex $q$ in $\Gamma_F$, the sign of $q$ in $\Gamma_F$ is equal to the sign of $q$ in $\Gamma$.
\end{enumerate}
By Definition \ref{d2}, this sub-multigraph $\Gamma_F$ describes the induced $S^1$-action on $F$ with induced unitary structure in the sense of Lemma \ref{l32}. For an example, see Example \ref{ex4}.

From the signs and weights at the fixed points we can determine the Chern numbers of $M$, because we can compute the Chern number of $M$ in terms of the signs and weights at the fixed points, in the Atiyah-Bott-Berline-Vergne localization formula, Theorem \ref{t25} \cite{AB}. Moreover, by Theorem \ref{t21}, we can determine the Hirzebruch $\chi_y$-genus of $M$. Because a multigraph encodes the signs and weights at the fixed points, these invariants are also encoded in a multigraph. Thus, considering what possible multigraphs are for a given unitary $S^1$-manifold can determine these invariants of the manifold, and can be useful for limiting possible examples.

We discuss the classification of circle actions on compact unitary manifolds, when the number of fixed points is small. First, if the circle acts on a compact unitary manifold $M$ with $1$ fixed point $p$, the manifold must be the point itself, that is, $M=\{p\}$. Second, if there are 2 fixed points, the following result is known.

\begin{theo} \cite{Ko2} \label{t313}
Let the circle act on a compact unitary manifold $M$ with 2 fixed points. Let $p$ and $q$ be the fixed points. Then exactly one of the following holds:
\begin{enumerate}[(1)]
\item $\dim M=2$, $\epsilon(p)=\epsilon(q)$, and the weight at $p$ and $q$ are $a$ and $-a$, respectively, for some positive integer $a$.
\item $\dim M=6$, $\epsilon(p)=\epsilon(q)$, and the weights at $p$ and $q$ are $\{-a-b,a,b\}$ and $\{-a,-b,a+b\}$, respectively, for some positive integers $a$ and $b$.
\item $\epsilon(p)=-\epsilon(q)$ and the weights at $p$ and $q$ agree up to order.
\end{enumerate}
\end{theo}

In any of the cases in Theorem \ref{t313}, there exists a manifold; the sphere $S^{2n}$ is such an example. We provide examples of manifolds in Theorem \ref{t313} and multigraphs describing these manifolds.

\begin{figure}
\begin{subfigure}[b][5.7cm][s]{.29\textwidth}
\centering
\vfill
\begin{tikzpicture}[state/.style ={circle, draw}]
\node[state] (a) at (0,0) {$p,+$};
\node[state] (b) at (0,4) {$q,+$};
\path (a) [->] edge node[left] {$a$} (b);
\end{tikzpicture}
\vfill
\caption{$\dim M=2$, $\epsilon(p)=\epsilon(q)=1$}\label{fig1-1}
\vspace{\baselineskip}
\end{subfigure}\qquad
\begin{subfigure}[b][5.7cm][s]{.29\textwidth}
\centering
\vfill
\begin{tikzpicture}[state/.style ={circle, draw}]
\node[state] (a) at (0,0) {$p,+$};
\node[state] (b) at (0,4) {$q,+$};
\path (a) [->] [bend left =20]edge node[left] {$a$} (b);
\path (a) [->] edge node[left] {$b$} (b);
\path (b) [->] [bend left =20]edge node[right] {$a+b$} (a);
\end{tikzpicture}
\vfill
\caption{$\dim M=6$, $\epsilon(p)=\epsilon(q)=1$}\label{fig1-2}
\vspace{\baselineskip}
\end{subfigure}\qquad
\begin{subfigure}[b][5.7cm][s]{.29\textwidth}
\centering
\vfill
\begin{tikzpicture}[state/.style ={circle, draw}]
\node[state] (a) at (0,0) {$p,+$};
\node[state] (b) at (0,4) {$q,-$};
\path (a) [->] [bend left =35]edge node[left] {$a_1$} (b);
\path (a) [->] [bend left =20]edge node[right] {$a_2 \cdots$} (b);
\path (a) [->] [bend right =35]edge node[right] {$a_n$} (b);
\end{tikzpicture}
\vfill
\caption{$\dim M=2n$, $\epsilon(p)=-\epsilon(q)=1$}\label{fig1-3}
\vspace{\baselineskip}
\end{subfigure}\caption{2 fixed points}\label{fig1}\qquad
\end{figure}

\begin{exa} \label{ex1}
Consider $S^2$ as an (almost) complex manifold. Let the circle act on $S^2$ by rotating $a$-times. The north pole $p$ and the south pole $q$ are the fixed points, and the weight at the north pole is $+a$ and the weight at the south pole is $-a$. Figure \ref{fig1-1} describes $S^2$ with this action.
\end{exa}

\begin{exa} \label{ex2}
Regarded as $G(2)/SU(3)$, the 6-sphere $S^6$ admits an almost complex structure, and it admits a circle action with 2 fixed points that preserves the almost complex structure. The weights at the fixed points are $\{-a-b,a,b\}$ and $\{-a,-b,a+b\}$ for some positive integers $a$ and $b$. Figure \ref{fig1-2} describes $S^6$ with this action.
\end{exa}

\begin{exa} \label{ex3} Let $n$ be any positive integer. Let $S^{2n}=\{(z_1,\cdots,z_n,x) \in \mathbb{C}^n \times \mathbb{R} \, | \, x^2+\sum_{i=1}^n |z_i|^2=1\}$ be the $2n$-dimensional sphere. There exists a complex structure on the vector bundle $TS^{2n} \oplus \underline{\mathbb{R}}^2=S^{2n} \times \mathbb{R}^{2n+2}=S^{2n} \times \mathbb{C}^{n+1}$ and hence $S^{2n}$ is unitary. Let $S^1$ act on $S^{2n}$ by
\begin{center}
$g \cdot (z_1,\cdots,z_n,x)=(g^{a_1} z_1, \cdots, g^{a_n} z_n, x)$
\end{center}
for all $g \in S^1$ and $(z_1,\cdots,z_n,x) \in S^{2n}$, where $a_i$ are positive integers, $1 \leq i \leq n$. There are two fixed points, $p=(0,\cdots,0,1)$ and $q=(0,\cdots,0,-1)$ with $\epsilon(p)=+1$ and $\epsilon(q)=-1$. Both of the fixed points have weights $\{a_1,\cdots,a_n\}$. Therefore, Figure \ref{fig1-3} describes $S^{2n}$ with this action. If some $a_i$ is instead a negative integer, the direction of the corresponding edge is reversed. Note that if $n \neq 1, 3$, $S^{2n}$ does not admit an almost complex structure. Figure \ref{fig1-3} cannot occur as a multigraph describing a compact almost complex $S^1$-manifold since the fixed point $q$ has $\epsilon(q)=-1$. \end{exa}

\begin{figure}
\begin{subfigure}[b][6.2cm][s]{.46\textwidth}
\centering
\vfill
\begin{tikzpicture}[state/.style ={circle, draw}]
\node[state] (a) at (0,0) {$p,+$};
\node[state] (b) at (0,4) {$q,+$};
\path (a) [->] [bend left =35]edge node[left] {$2$} (b);
\path (a) [->] [bend left =15]edge node[right] {$3$} (b);
\path (a) [->] [bend right =10]edge node[right] {$5$} (b);
\path (a) [->] [bend right =35]edge node[right] {$6$} (b);

\end{tikzpicture}
\vfill
\caption{Multigraph describing $S^8$}\label{fig2-1}
\vspace{\baselineskip}
\end{subfigure}\qquad
\begin{subfigure}[b][6.2cm][s]{.46\textwidth}
\centering
\vfill
\begin{tikzpicture}[state/.style ={circle, draw}]
\node[state] (a) at (0,0) {$p,+$};
\node[state] (b) at (0,4) {$q,-$};
\path (a) [->] [bend left =20]edge node[right] {$3$} (b);
\path (a) [->] [bend right =20]edge node[right] {$6$} (b);
\end{tikzpicture}
\vfill
\caption{(Sub)multigraph describing $F=S^4=(S^8)^{\mathbb{Z}_3}$}\label{fig2-2}
\vspace{\baselineskip}
\end{subfigure}\caption{Isotropy submanifold and its (sub)multigraph}\label{fig2}\qquad
\end{figure}
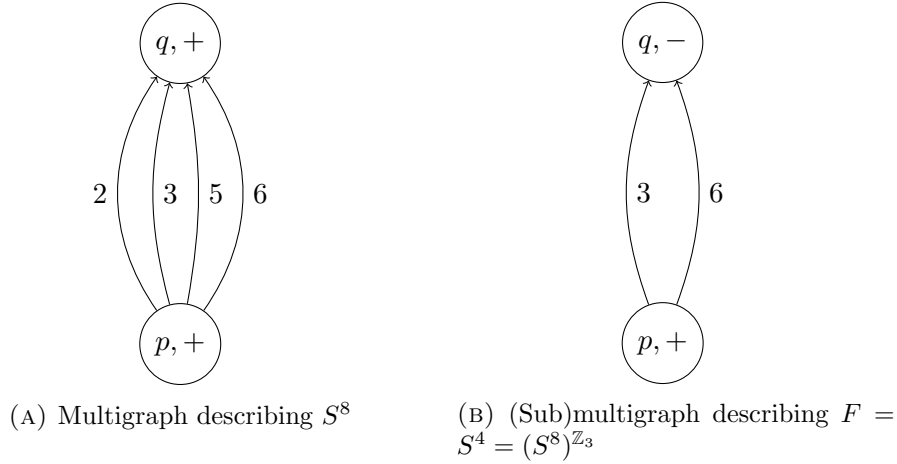

In below, we shall explain in an example that for a multigraph describing a unitary $S^1$-manifold, its sub-multigraph describes its isotropy submanifold.

\begin{exa} \label{ex4} In Example \ref{ex2}, take $n=4$, $a_1=2$, $a_2=3$, $a_3=5$, and $a_4=6$. Then Figure \ref{fig2-1} describes this action on $S^8$. We shall consider the set $F:=(S^8)^{\mathbb{Z}_3}$, the set of points in $S^8$ that are fixed by the $\mathbb{Z}_3$-action. From the action, it follows that 
\begin{center}
$F=\{(0,z_1,0,z_2,x) \in \mathbb{C}^4 \times \mathbb{R} \, | \, x^2+|z_2|^2+|z_4|^2=1\}$.
\end{center}
The unitary structure on $M$ through the identification $TS^8 \times \underline{\mathbb{R}}^2 = S^8 \times \mathbb{C}^5$ induces a unitary structure on $F$ through the identification $TF \oplus \underline{\mathbb{R}}^2|_F=TS^4 \times \underline{\mathbb{R}}^2|_F=S^4 \times \mathbb{R}^6=S^4 \times \mathbb{C}^3$, and thus $F$ is a smaller dimensional compact unitary submanifold, as proved in Lemma \ref{l23}. The circle action on $M$ restricts to act on $F$, and this action on $F$ also has $p$ and $q$ as fixed points with weights $\{3,6\}$ for both fixed points. In addition, with this unitary structure on $F$, $\epsilon_F(p)=\epsilon_M(p)=+1$ and $\epsilon_F(q)=\epsilon_M(q)=-1$; also see Lemma \ref{l32}. Then Figure \ref{fig2-2} describes this action on $F$, which is a sub-multigraph of Figure \ref{fig2-1} consisting of edges in Figure \ref{fig2-1} whose labels are multiples of 3.
\end{exa}

\bibliographystyle{alpha}
\bibliography{jang}

\end{document}